\documentclass[12pt]{amsart}
\usepackage{amssymb, amsmath}
\title{New classes of groups which are equational domains }

\author{\sc O. Al-Raisi}
\address{O. Al Raisi: Department of  Mathematics,  College of Science, Sultan Qaboos University, Muscat, Oman}
\email{omartalibmiran@gmail.com}

\author{M. Shahryari}
\address{M. Shahryari: Department of  Mathematics,  College of Science, Sultan Qaboos University, Muscat, Oman}
\email{m.ghalehlar@squ.edu.om}

\newtheorem{corollary}{Corollary}

\newtheorem {theorem}{Theorem}

\numberwithin{equation}{section}

\newcommand{\CSNk}{\mathrm{CSN}_k}

\begin{document}

\maketitle
\begin{abstract}
A group is CSA, if all of its maximal abelian subgroups are malnormal. It is known that every non-abelian  CSA group is an equational domain. We generalize this result in two directions: we show that for  a non-nilpotent group $G$ and a fixed positive integer $k$, if all maximal elements in the set of class $k$ nilpotent subgroups of $G$ are malnormal, then $G$ is an equational domain. Also, we prove that if a group $G$ is not locally nilpotent and if every maximal locally nilpotent subgroup of $G$ is malnormal, then $G$ is an equational domain.
\end{abstract}
\vspace{1cm}


\vspace{1cm}

A group $G$ is called CSA (conjugately separated abelian), if every maximal abelian subgroup  of $G$ is {\em malnormal}. This means that, if $H$ is a maximal abelian subgroup of $G$ and $x\in G\setminus H$, then $H\cap H^x=1$. The class of CSA groups is quite wide and is of considerable importance in the study of residually free groups, universal theory of non-abelian free groups, limit groups, exponential groups and equational domains in algebraic geometry over groups (see \cite{BMR}, \cite{Champ}, \cite{MR1} , and \cite{MR2}).

In \cite{Shah}, the second author initiated the study of $\CSNk$ groups which are natural generalizations of CSA groups. A group $G$ is $\CSNk$ (conjugately separated nilpotent of class $k$) if and only if every maximal element in the set of class $k$ nilpotent subgroup of $G$ is malnormal. The case $k=1$  obviously coincides with ordinary CSA groups. It is shown in \cite{Shah}  that the class of $\CSNk$ groups  share many similar properties with the classical case of  CSA groups.

A group $G$ is called an {\em equational domain} if and only if, for every positive integer $n$, the union of two algebraic sets in $G^n$ is also an algebraic set. In \cite{BMR}, it is shown that every non-abelian  CSA group is an equational domain. We prove that this result is a special case of a general theorem: if a non-nilpotent group is $\CSNk$, then it is an equational domain. Furthermore, we show that if a group $G$ is not locally nilpotent but all of its maximal locally nilpotent subgroups are malnormal, then $G$ is an equational domain. Many examples of $\CSNk$ groups are given in \cite{Shah}; a method of constructing more examples will be provided here. Starting with a free nilpotent group $A$ in the variety of class $k$ nilpotent groups, we consider an $A$-free group $F$ (which is by definition, a free product of some copies of $A$) and we prove that if $Y\subseteq F^n$ is an irreducible algebraic set, then its coordinate group $\Gamma(Y)$ is $\CSNk$.

In what follows, we use these notations: all simple commutators $[ x_1, x_2, \ldots, x_{k+1}]$ are left aligned. Given a group $G$ and a subset $X\subseteq G$, the subgroup of $G$ generated by $X$ will be denoted by $\langle X\rangle$. A conjugate $a^x$ (or $H^x$) is $x^{-1}ax$ (similarly, $x^{-1}Hx$). By a nilpotent group of class $k$, we mean a group in which the identity $[ x_1, x_2, \ldots, x_{k+1}]=1$ holds.

\section{Algebraic sets and equational domains}
We use the same notations as in \cite{BMR} and \cite{DMR}. Suppose $G$ is a group and $X=\{x_1, x_2, \ldots, x_n\}$ is a set of variables. Let $\mathbb{F}[X]$ be the free group generated by $X$ and $G[X]=G\ast \mathbb{F}[X]$ be the free product of $G$ and $\mathbb{F}[X]$. Every element $G[X]$ is a group word in variables $x_1, x_2, \ldots, x_n$ with coefficients from $G$. If $w(x_1, \ldots, x_n)\in G[X]$, then $w(x_1, \ldots, x_n)\approx 1$ is called a group equation. A group $H$ containing an isomorphic copy of $G$ is termed a $G$-group. For a given equation $w(x_1, \ldots, x_n)\approx 1$, the set
$$
\{ (h_1, \ldots, h_n)\in H^n:\ w(h_1, \ldots, h_n)=1\}
$$
is the solution set of the given equation in $H$. A system of equations with coefficients from $G$ is any set of equations $S\approx 1$, where $S\subseteq G[X]$. The {\em algebraic set} corresponding to this system is the set of all common solutions of all equations in $S\approx 1$ in $H^n$. We denote this algebraic set by $V_H(S)$.

A topology can be defined on $H^n$ using the collection of algebraic sets as a subbasis for the closed sets; thus, algebraic sets, finite unions of algebraic sets, and  arbitrary intersections of finite unions of algebraic sets are all closed sets. This is called the Zariski topology on $H^n$. This topology is Noetherian if and only if for every $S\subseteq G[X]$, there exists a finite subset $S_0\subseteq S$ such that $V_H(S)=V_H(S_0)$. In this case, we say that the group $H$ is $G$-{\em equationally Noetherian}. If $H$ is $G$-equationally Noetherian, then every algebraic set can be decomposed uniquely as a finite union of {\em irreducible} algebraic sets. The group $H$ is called a $G$-domain if and only if for any natural number $n$, the union of every two algebraic sets in $H^n$ is again an algebraic set. In this case, every closed set in the Zariski topology is an algebraic set. There is another definition for the concept of $G$-domain in terms of {\em zero divisors}. An element $x\in H$ is called a zero divisor, if there exists a non-identity element $y\in H$ such that for every $g\in G$, the equation $[x^g, y]=1$ holds. In \cite{DMR}, it is proved that a $G$-group $H$ is $G$-domain if and only if $H$ does not contain any non-trivial zero divisor. It is not hard to see that $H$ is a $G$-domain if and only if it satisfies the following property: for every non-trivial subgroup $K\leq H$, if $G$ normalizes $K$ (i.e., $G\subseteq N_H(K)$), then the centralizer $C_H(K)$ is trivial. This becomes more interesting in the case when $H=G$; that is, the case of {\em Diophantine geometry} over $G$. In this Diophantine case, we use the phrases "equationally Noetherian" and "domain" instead of "$G$-equationally Noetherian" and "$G$-domain", respectively. As a result, $G$ is a domain if and only if for every non-trivial normal subgroup $K\leq G$, we have $C_G(K)=1$. Looking at domains from this viewpoint, one can see immediately that a non-abelian finite group is a domain if and only if it is monolithic, that is, it has a unique minimal normal subgroup. In the case of infinite groups, besides monolithic groups, there are, of course, many other groups which are domains. As an example, every non-abelian free group is a domain. This can be generalized to CSA groups. Let $G$ be a CSA group and $K$ be a non-trivial normal subgroup of $G$; let $C_G(K)\neq 1$ and choose a non-trivial element $x\in G$ such that $[x, K]=1$. As every CSA group is {\em commutative transitive} (that is, the binary relation $[a, b]=1$ is an equivalence relation on the set of non-trivial elements of $G$), the subgroup $K$ must be abelian. Let $M$ be a maximal abelian subgroup of $G$ containing $K$. This subgroup $M$ must be malnormal. But, as $K$ is a normal subgroup, for any $g\in G$ we have $K\subseteq M\cap M^g$, which is a contradiction. This shows that every CSA group is a domain (see also \cite{BMR}).  As we mentioned in the introduction, our aim is to generalized this result to a wider class of groups. We only consider the most interesting case; namely, that of Diophantine geometry over a group $G$. Before proceeding however, one more concept needs to be defined. Let $Y\subseteq G^n$ be an arbitrary subset. A normal subgroup of $G[X]$ can be defined as follows:
$$
\mathrm{Rad}(Y)=\{ w\in G[X]:\ Y\subseteq V_G(w\approx 1)\}.
$$
This is the {\em radical } of $Y$ over $G$. The quotient group
$$
\Gamma(Y)=G[X]/\mathrm{Rad}(Y)
$$
is called the {\em coordinate group} of $Y$. In the next section, we will need the following {\em unification theorem} from \cite{BMR}.

\begin{theorem}
Let $G$ be an equationally Noetherian domain and $Y\subseteq G^n$ be an  algebraic set. Then the following are equivalent:
\begin{enumerate}
\item{$Y$ is irreducible}
\item{$\Gamma(Y)$ is $G$-equationally Noetherian $G$-domain}
\item{$\Gamma(Y)$ is fully residually $G$ as a $G$-group}
\item{$\Gamma(Y)$ has the same universal theory as $G$}
\end{enumerate}
\end{theorem}

\section{Main results}
The notion of a CSA group is generalized in \cite{Shah} as follows: Let $k$ be a fixed positive integer and $G$ be a group. Suppose all maximal elements in the family of class $k$ nilpotent subgroups of $G$ are malnormal. Then we say that $G$ is $\CSNk$. It is proved that this new class shares many properties with the class of CSA groups. As an example, it is shown that every $\CSNk$ group is $\mathrm{NT}_k$: if $H_1$ and $H_2$ are class $k$ nilpotent subgroups of $G$ with non-trivial intersection, then $\langle H_1, H_2\rangle$ is also nilpotent of class $k$. Our first result shows that non-nilpotent $\CSNk$ groups are domains.

\begin{theorem}
Suppose $G$ is $\CSNk$ but not nilpotent. Then $G$ is a domain.
\end{theorem}

\begin{proof}
If $G$ is not a domain, then it contains a non-trivial normal subgroup $N$ such that $C_G(N)\neq 1$. Let $y\in C_G(N)$ be a non-identity element and $x_1, x_2, \ldots, x_{k+1}\in N$ be arbitrary. For every $i$, the subgroup $\langle x_i, y\rangle$ is abelian (and so, nilpotent of class $k$). Therefore, for any $i$ and $j$, both subgroups $\langle x_i, y\rangle$ and $\langle x_j, y\rangle$ are nilpotent of class $k$, and they have a non-trivial intersection. By the aforementioned, the subgroup $\langle x_i, x_j, y\rangle$ must be nilpotent of class $k$. Repeating this argument, we infer that the subgroup $\langle x_1, x_2, \ldots, x_{k+1}, y\rangle$ is nilpotent of class $k$ and hence, $[ x_1, x_2, \ldots, x_{k+1}]=1$; thereby, establishing that $N$ is nilpotent of class $k$. Now, let $M$ be a maximal class $k$ nilpotent subgroup containing $N$. For each $g\in G$, we have $1\neq N\subseteq M\cap M^g$ and by the malnormality of $M$, we must have $g\in M$. This means that $G$ is nilpotent which is a contradiction. Hence, $G$ is a domain.
\end{proof}

Let $A$ be a free nilpotent group of class $k$. In \cite{Shah}, the concept of an $A$-free group is defined: A group $F$ is $A$-free if and only if it is the free product of a set of copies of $A$. The case $A=\mathbb{Z}$ coincides with the concept of ordinary free group. In \cite{Shah} it is proved that every (fully residually) $A$-free group is $\CSNk$. As a result, we have the following.

\begin{corollary}
Let $A$ be a finitely generated free nilpotent group of class $k$ and $F$ be an $A$-free group. For a positive integer $n$, let $Y\subseteq F^n$be an irreducible algebraic set. Then the coordinate group $\Gamma(Y)$ is a $\CSNk$ group.
\end{corollary}

\begin{proof}
By the previous theorem, $F$ is a domain. Our aim is to show that it is also equationally Noetherian.
A theorem of Jennings (see \cite{Jen}) shows that every finitely generated torsion-free nilpotent group embeds into $\mathrm{UT}(r, \mathbb{Z})$, the group of unipotent upper triangular integral matrices for some $r$. Hence $A$ embeds into $\mathrm{UT}(r, \mathbb{Z})$ and therefore, in $\mathrm{GL}_r(\mathbb{Q})$. As the degree of the algebraic closure $\overline{\mathbb{Q}}$ over $\mathbb{Q}$ is infinite, we can use the main result of \cite{Mar} to conclude that $A\ast A$ embeds in $\mathrm{GL}_N(\mathbb{Q}(t))$, for some positive integer $N$ and a variable $t$. This shows that $A\ast A$ is linear and hence, equationally Noetherian. Since $F$ embeds in $A\ast A$ (see \cite{Shah}), it is equationally Noetherian, as well. Now, we can apply Theorem 1 to see that $\Gamma(Y)$ is fully residually $A$-free and hence, by \cite{Shah}, it is $\CSNk$.
\end{proof}

In \cite{Shah}, another class of groups is introduced which also generalizes CSA groups. This is the class of all groups in which all maximal locally nilpotent subgroups are malnormal. It is shown that in such a group,  for every pair $K_1$ and $K_2$ of locally nilpotent subgroups, $K_1\cap K_2\neq 1$ implies that $\langle K_1, K_2\rangle$ is also locally nilpotent. It is also mentioned in \cite{Shah} that many properties of this new kind of groups coincide with those of $\CSNk$ groups. Below, we prove that every non-locally nilpotent group of this kind is also a domain.

\begin{theorem}
Let $G$ be a group which is not locally nilpotent. Let every maximal locally nilpotent subgroup of a group $G$ be malnormal. Then $G$ is a domain.
\end{theorem}

\begin{proof}
Suppose $G$ is not a domain. Then, it contains a non-trivial normal subgroup $N$ such that $C_G(N)\neq 1$; hence, there is a non-identity element $a\in G$ such that $[a, N]=1$. Let $x_1, x_2, \ldots, x_n\in N$ be arbitrary elements. For, each $i\in\{1,...,n\}$, let $K_i$ denote the abelian group $\langle a, x_i\rangle$ and consider two such groups $K_i=\langle a, x_i\rangle$ and $K_j=\langle a, x_j\rangle$. As the union of any chain of locally  nilpotent subgroups is locally nilpotent, there is, by Zorn's lemma, a maximal locally nilpotent subgroup $H$ containing $K_i$. We claim that $K_j\subseteq H$. If not, there is an element $z\in K_j\setminus H$ and, by malnormality of $H$, we have $H\cap H^z=1$. We know that $\langle a, z\rangle \subseteq K_j$ and hence $[a, z]=1$. This means that $a^{-1}z^{-1}az=1$ and so $z^{-1}az=a\in H$. As $H$ is malnormal, we must have $z\in H$; a contradiction. Therefore, our claim that $K_j\subseteq H$ is proved. Now, $K_i, K_j\subseteq H$ implies that the subgroup $\langle a, x_i, x_j\rangle$ is nilpotent. Applying this argument to the subgroups
$$
K_{ij}=\langle a, x_i, x_j\rangle, \ K_{r s}=\langle a, x_r, x_s\rangle.\,
$$
we conclude that the bigger subgroup $\langle a, x_i, x_j, x_r, x_s\rangle$ is nilpotent. Finally, continuing this argument, we conclude that the subgroup
$$
\langle a, x_1, x_2, \ldots, x_n\rangle
$$
is nilpotent and this means that $N$ is locally nilpotent. Using Zorn's lemma once again,  we may choose a maximal locally nilpotent subgroup $M$ containing $N$. For every $g\in G$, we have
$$
1\neq N\subseteq M\cap M^g,
$$
and, since $M$ is malnormal, we get $g\in M$; that is, $M=G$. But this is impossible as $G$ was hypothesised to be not locally nilpotent. This shows that $G$ is a domain which ends the proof.
\end{proof}

\end{document}